\documentclass[12pt,oneside]{amsart}
\usepackage{amsmath}
\usepackage{amsthm}
\usepackage{amsfonts}
\usepackage{amssymb}
\usepackage{enumerate}

\newtheorem{theorem}{Theorem}

\newtheorem{proposition}[theorem]{Proposition}
\newtheorem{corollary}[theorem]{Corollary}

\theoremstyle{definition}

\newtheorem{example}[theorem]{Example}

\theoremstyle{remark}
\newtheorem{remark}[theorem]{Remark}
\renewcommand{\phi}{\varphi}


\begin{document}

\title[Symmetric powers of complex manifolds]{Function theoretic properties of symmetric powers of complex manifolds}


\address{Institute of Mathematics, Faculty of Mathematics and Computer Science, Jagiellonian
University,  \L ojasiewicza 6, 30-348 Krak\'ow, Poland}

\author{W\l odzimierz Zwonek}
\email{wlodzimierz.zwonek@uj.edu.pl}

\thanks{The author was partially  supported by the
OPUS grant no. 2015/17/B/ST1/00996 financed by the National
Science Centre, Poland.}

\begin{abstract} In the paper we study properties of symmetric powers of complex manifolds. We investigate a number of function theoretic properties (e. g. (quasi) $c$-finite compactness, existence of peak functions) that are preserved by taking the symmetric power. The case of symmetric products of planar domains is studied in a more detailed way. In particular, a complete description of the Carath\'eodory and Kobayashi hyperbolicity and Kobayashi completeness in that class of domains is presented.
\end{abstract}

\maketitle


\section{Introduction}

Let $X$ be a connected complex manifold of dimension $m$. We define its {\it $n$-th symmmetric power $X^n_{sym}$} as the quotient $X^n$ under the 
action of the group of all permutations of $\{1,\ldots,n\}$. Recall that $X_{sym}^n$ has the structure of a complex analytic space. In the case when $m=1$ the space $X^n_{sym}$ is actually a complex manifold. If $X=D\subset\mathbb C$ is a domain then we have a realization of $D_{sym}^n$ as a domain in $\mathbb C^n$. More precisely, its biholomorphic realization is the following \textit{$n$-dimensional symmetrization} (or {\it symmetric product of planar domains})
\begin{equation}
S_n(D):=\pi_n(D^n),
\end{equation}
where $\pi_n:\mathbb C^n\mapsto\mathbb C^n$ is {\it the symmetrization map} (the $j$-th coordinate is the $j$-th elementary symmetric polynomial). 
In other words $\pi_{n,j}(\lambda_1,\ldots,,\lambda_n)=\sigma_j(\lambda_1,\ldots,\lambda_n)$, $\lambda_j\in\mathbb C$, $j=1,\ldots,n$, where $\sigma_j$ satisfies the equality
\begin{equation}
(\lambda-\lambda_1)\cdot\ldots\cdot (\lambda-\lambda_n)=\lambda^n+\sum_{j=1}^n(-1)^j\sigma_j(\lambda_1,\ldots,\lambda_n)\lambda^{n-j},\;\lambda\in\mathbb C.
\end{equation}
As to the background on basic properties of symmetric powers we refer the Reader e. g. to \cite{Whi 1972}.

\subsection{Description of results} 
In the paper we present a number of properties of $X^n_{sym}$. Some of them may be obtained from general properties of the realization of $\mathbb D^n_{sym}$ as a domain, i. e. the so called {\it the symmetrized polydisc $\mathbb G_n:=S_n(\mathbb D)$}, ($\mathbb D$ denotes the unit disc in $\mathbb C$). The last domain has been extensively studied in the last two decades (see for instance \cite{Agl-You 2004}, \cite{Cos 2004}, \cite{Edi-Zwo 2005}, \cite{Jar-Pfl 2013} and references there).

The starting point for considerations in the paper were inspired by recent developments on the function theory in symmetric powers (see e. g. \cite{Cha-Gor 2015}, \cite{Bha-Bis-Div-Jan 2018} and \cite{Cha-Grow 2018}).

First we present a more general result to that of Theorem 1.4 in \cite{Bha-Bis-Div-Jan 2018} where the proof of the Kobayashi completeness of symmetric powers of some of Riemann surfaces relies on the proof of existence of peak functions. Similarly as in \cite{Bha-Bis-Div-Jan 2018} the presentation below actually deals with a stronger version of completeness - the $c$-finite compactness and shows that the notion (more precisely, the weaker notion of quasi $c$-finite compactness) is preserved under taking the symmetric power, which is done in a general case of complex manifolds  (Theorem~\ref{theorem:c-finitely-compact}). Following the same line of argument relying upon analoguous results in the symmetrized polydisc we present a result on the existence of peak functions in symmetric powers (Theorem~\ref{theorem:peak-function}).

In Section~\ref{section:planar-domains} we concentrate on properties of symmetric products of planar domains in $\mathbb C$. We show the linear convexity of such domains (Proposition~\ref{proposition:linearly-convex}), we present a Riemann-type mapping theorem for them (Theorem~\ref{theorem:riemann-like}) and then we discuss to which extent the Kobayashi hyperbolicity (completeness) is preserved under taking the symmetric power -- a complete description of Kobayashi hyperbolicity (completeness) in that class is given in Theorem~\ref{theorem:kobayashi-complete}. Finally, we present a result on preserving the Carath\'eodory hyperbolicity under taking the symmetric powers of planar domains (Proposition~\ref{proposition:caratheodory-hyperbolic}).

\section{General case}\label{section:general-case}
In this Section we present results for a general class of symmetric powers of manifolds.

\subsection{(Quasi) $c$-finite compactness}

Let us recall that any holomorphic mapping $f:X\to Y$ ($X$ and $Y$ are complex manifolds, not necessarily of the same dimension) induces a holomorphic mapping $F_f:X^n_{sym}\to Y^n_{sym}$ by the formula (a  typical element of $X^n_{sym}$ generated by $(z_1,\ldots,z_n)\in X^n$ is denoted by $\langle z_1,\ldots,z_n\rangle$)
\begin{equation}
F_f(\langle z_1,\ldots,z_n\rangle):=\langle f(z_1),\ldots,f(z_n)\rangle,\; z_j\in X,\; j=1,\ldots,n.
\end{equation}
In the case $Y=\mathbb D$ we have the holomorphicity of the mapping
\begin{equation}
X^n_{sym}\owns \langle z_1,\ldots, z_n\rangle\to \pi_n(f(z_1),\ldots,f(z_n)\rangle\in\mathbb G_n.
\end{equation}

We should also be aware of the fact that any holomorphic function $F:X^n_{sym}\to\mathbb D$ may be identified with a symmetric function
$\tilde F:X^n\to\mathbb D$ -- a function $\tilde F\in\mathcal O(X^n,\mathbb D)$ is called {\it symmetric} if $\tilde F(z_{\sigma(1)},\ldots,z_{\sigma(n)})=\tilde F(z_1,\ldots,z_n)$ for any permutation $\sigma$ of $\{1,\ldots,n\}$, $z_1,\ldots,z_n\in X$. The identification is given by the relation
\begin{equation}
F(\langle z_1,\ldots,z_n\rangle)=\tilde F(z_1,\ldots,z_n),\; z_j\in X,; j=1,\ldots,n.
\end{equation}
The space of symmetric holomorphic functions $X^n\to\mathbb D$ is denoted by $\mathcal O_s(X^n,\mathbb D)$.

The last observation lets us define {\it the Carath\'eodory pseudodistance} $c_{X^n_{sym}}$ as follows:
\begin{multline}
c_{X^n_{sym}}(\langle z_1,\ldots,z_n\rangle,\langle w_1,\ldots,w_n\rangle)=\\
\sup\left\{p(0,F(w_1,\ldots,w_n)):F\in\mathcal O_s(X^n,\mathbb D), F(z_1,\ldots,z_n)=0\right\},
\end{multline}
where $p$ is the Poincar\'e distance on $\mathbb D$.

If, on some complex structure $X$ (e. g. a complex manifold), we may well-define the Carath\'eodory pseudodistance we call $X$ {\it quasi $c$-finitely compact } if  for any sequence $(z_k)_k\subset X$ without the accummulation point we have $c_X(z_1,z_k)\to\infty$.  Recall that if $X$ is additionally {\it Carath\'eodory hyperbolic }, i. e. $c_X(w,z)>0$, $w,z\in X$, $w\neq z$ then $X$ is called {\it $c$-finitely compact}. As to the basic properties related to the Carath\'eodory pseudodositance (as well as to other holomorphically invariant functions) we refer the Reader to e. g. \cite{Kob 1998}, \cite{Jar-Pfl 2013}.
 
As we shall see below a natural property that is inherited by the symmetric power is the quasi $c$-finite compactness.



\begin{theorem}\label{theorem:c-finitely-compact} Let $X$ be a connected complex manifold. Then $X^n_{sym}$ is quasi $c$-finitely compact iff $X$ is quasi $c$-finitely compact.
\end{theorem}
\begin{proof} Assume that $X$ is quasi $c$-finitely compact.
Fix $\langle z_1,\ldots,z_n\rangle\in X^n_{sym}$. Let $\left(\langle w_1^k,\ldots,w_n^k\rangle\right)_k\subset X^n_{sym}$ be a sequence without an accummulation point. Then withot loss of generality we may assume that $(w_1^k)_k$ has no accummulation point in $X$. Let $f_k\in\mathcal O(X,\mathbb D)$ be such that $f_k(z_1)=0$ and $f_k(w_1^k)\to 1$. Then the contractivity of the Carath\'eodory pseudodistance gives the following
\begin{multline}
c_{X^n_{sym}}(\langle z_1,\ldots,z_n\rangle,\langle w_1^k,\ldots,w_n^k\rangle)\geq \\
c_{\mathbb G_n}(\pi_n(f_k(z_1),\ldots,f_k(z_n)),\pi_n(f_k(w_1^k),\ldots,f_k(w_n^k)).
\end{multline}
Since $\pi_n(f_k(w_1^k),\ldots,f_k(w_n^k))\to\partial \mathbb G_n$ and the set $\{\pi_n(f_k(z_1),\ldots,f_k(z_n)):k=1,2,\ldots\}$ is relatively compact in $\mathbb G_n$, the $c$-finite compactness of $\mathbb G_n$ (see \cite{Nik-Pfl-Zwo 2008}) gives the convergence of the above expression to infinity which finishes the proof.

Assume now that $X^n_{sym}$ is quasi $c$-finitely compact. Fix $\langle z_1,\ldots,z_n\rangle\in X^n_{sym}$. Let $(w_1^k)_k\subset X$ be a sequence without accumulation point. Then the sequence $\left(\langle w_1^k,z_2,\ldots,z_n\rangle\right)_k$ has no accumulation point, either. Then the inequalities
\begin{equation}
c_X(z_1,w_1^k)\geq c_{X^n_{sym}}(\langle z_1,\ldots,z_n\rangle,\langle w_1^k,z_2,\ldots,z_n\rangle)\to \infty
\end{equation}
show the quasi $c$-finite compactness of $X$.
\end{proof}

\begin{remark} In the case when $X$ is a bounded domain in $\mathbb C$ the above theorem is a reformulation of Theorem 4.1 in \cite{Kos-Zwo 2013} (applied to the proper holomorphic mapping $(\pi_n)_{|D^n}:D^n\to S_n(D)$). We also should be aware of the fact that the idea of the proof of the above theorem is exactly the same as that of Theorem 4.1 in \cite{Kos-Zwo 2013}.
\end{remark}
 
\begin{remark} In Theorem 1.4 in \cite{Bha-Bis-Div-Jan 2018} a result on Kobayashi completeness of symmetric powers of some Riemann surfaces is formulated. The proof relies on the existence of some peak functions together with the application of Result 3.6  from \cite{Bha-Bis-Div-Jan 2018}, in which the fact of $c$-finite compactness is claimed under assumption of the existence of some peak functions. It is however not explained us how the necessary fact of the Carath\'eodory hyperbolicity is obtained only with the help of the existence of peak functions (in the case studied in the reasoning from \cite{Kob 1998}, to which the paper \cite{Bha-Bis-Div-Jan 2018} appeals, the hyperbolicity is trivially satisfied). 

\end{remark}

\begin{remark} If $d$ denotes a family of holomorphically invariant functions (for instance the Carath\'eodory ($c$) or Kobayashi ($k$) pseudodistance) then $d$-hyperbolicity of a complex manifold $X^n_{sym}$ implies the $d$-hyperbolicity of $X$. Actually, fix $w_1,z_1\in X$, $w_1\neq z_1$. Choose $w_2\in X$. Then $\langle w_1,w_2,\ldots,w_2\rangle\neq \langle z_1,w_2,\ldots,w_2\rangle$.  Consequently,
\begin{equation}
c_X(w_1,z_1)\geq c_{X^n_{sym}}(\langle w_1,w_2,\ldots,w_2\rangle,\langle z_1,w_2,\ldots,w_2\rangle)>0.
\end{equation} 
As to the implication:

$X$ is $d$-hyperbolic $\implies$ $X^n_{sym}$ is $d$-hyperbolic

the observation in the next remark shows that it is not true in general.
\end{remark}


\begin{remark} It would be tempting to formulate a similar equivalence as in Theorem~\ref{theorem:c-finitely-compact} for the notion of the Kobayashi quasi completeness. However, 
the example  $\mathbb C\setminus\{0,1\}$, shows that the Kobayashi completeness of $X$ does not guarantee  any reasonable property of the Kobayashi pseudodistance of $X^n_{sym}$.

In fact, put $D:=\mathbb C\setminus\{0,1\}$, $n=2$. Then 
\begin{equation}
S_2(D)=\mathbb C^2\setminus(\mathbb C\times\{0\}\cup\{(\lambda+1,\lambda):\lambda\in\mathbb C\}). 
\end{equation}
The last is the space $\mathbb C^2$ with two complex lines intersected which is affinely isomorphic with $\mathbb C_*^2$ for which the Kobayashi pseudodistance vanishes. In the sequel we shall present a complete description of the Kobayashi hyperbolicity, Kobayashi completeness, Carath\'eodory hyperbolicity and $c$-finite compactness in the class of symmetric products of planar domains.
\end{remark}

\subsection{Peak functions}

In the paper \cite{Bha-Bis-Div-Jan 2018} the proof of the Kobayashi completeness (Theorem 1.4) was conducted with the help of the existence of peak functions (that was done for some Riemann surfaces).  We generalize the result and simplify the proof below. We also see that we may reduce the proof of the existence of peak functions in symmetric powers to the existence of some of peak functions in the original complex manifold.

For a domain $Y$ in a complex manifold $X$ and $K\subset\overline{Y}$ we define $A(Y,K):=\mathcal O(Y)\cap\mathcal C(Y\cup K)$. We call a point $z\in K$ an {\it $A(Y,K)$ peak point} if there is an $f\in A(Y,K)$ such that $|f|<1$ on $Y$ and $f(z)=1$. We call $f$ {\it the $A(Y,K)$ peak function at $z$}.

\begin{theorem}\label{theorem:peak-function} Let $Y$ be a domain in a complex manifold $X$. Assume that $z_1\in\partial Y$ is the $A(Y,\{z_1,\ldots,z_n\})$ peak point, where $z_2,\ldots,z_n\in\overline{Y}$. Then the point $\langle z_1,\ldots,z_n\rangle$ is an $A(Y^n_{sym},\{\langle z_1,\ldots,z_n\rangle\})$ peak point.
\end{theorem}
\begin{proof} Let $f$ be an $A(Y,\{z_1,\ldots,z_n\})$ peak function at $z_1$.
Let $F\in\mathcal O(\mathbb G_n\cup U)$, where $U$ is a neighborhood of $\pi_n(f(z_1),\ldots,f(z_n))\in\partial G_n$ be such that 
\begin{equation}
F(\pi_n(f(z_1),\ldots,f(z_n)))=1 \text{ and $|F|<1$ on $\mathbb G_n$}
\end{equation}
(see Theorem 2.1 in \cite{Kos-Zwo 2013}). The function \begin{equation}
Y^n_{sym}\cup\{\langle z_1,\ldots,z_n\rangle\}\owns \langle w_1,\ldots,w_n\rangle\to F(\pi_n(f(w_1),\ldots,f(w_n)))
\end{equation}
is the desired peaking function. 
\end{proof}

\begin{remark}
The assumption in Theorem~\ref{theorem:peak-function} is a weaker one than that in the proof of Theorem 1.4 of \cite{Bha-Bis-Div-Jan 2018}.
\end{remark}

\section{Remarks on symmetric products of planar domains}\label{section:planar-domains}
In this section we present properties of symmetric products of planar domains.

\subsection{General properties}

Recall that if $D$ is a domain in $\mathbb C$ then we have a nice representation of $D^n_{sym}$ as a domain in $\mathbb C$. We work therefore on this representation, i. e. the domain $S_n(D)$. First we collect some facts concerning $S_n(D)$.
Below we list some known or straightforward properties of $S_n(D)$.

\begin{remark} 
\begin{itemize}
\item $S_n(D)$ is a domain in $\mathbb C^n$ and $(\pi_n)_{|D^n}:D^n\to S_n(D)$ is proper onto the image,\\
\item if $\Sigma_n:=\{z\in D^n: z_j= z_k \text{ for some $j\neq k$}\}$ then $\pi_n:D^n\setminus \Sigma_n\to  S_n(D)\setminus\pi_n(\Sigma_n)$ is a holomorphic covering,\\
\item $S_n(D)$ is bounded iff $D$ is bounded,\\
\item $\overline{S_n(D)}=\pi_n(\overline{D}^n)$, $\partial S_n(D)=\pi_n(\partial D\times\overline{D}^{n-1})$,\\
\item if $D$ is additionally bounded then the mapping $(\pi_n)_{D^n}$ maps $A(D^n)$ peak points ($A(\Omega):=\mathcal O(\Omega)\cap\mathcal C(\overline{\Omega})$) onto $A(S_n(D))$ peak points. In particular, the Shilov boundary $\partial_S(S_n(D))$ equals $S_n(\partial_S(D))$ (see Theorem 3.1 in \cite{Kos-Zwo 2013}).
\end{itemize} 
\end{remark}

Recall that a domain
$\Omega\subset\mathbb C^n$ is called {\it linearly convex} if for any $w\in\mathbb C^n\setminus \Omega$ we may find an affine hyperplane $H$ passing through $w$ and disjoint from $\Omega$. Following step by step the idea from \cite{Nik-Pfl-Zwo 2008} we get the linear convexity of $S_n(D)$.

\begin{proposition}\label{proposition:linearly-convex} Let $D$ be a domain in $\mathbb C$, $n\geq 2$. Let $w=\pi_n(\mu)\in\mathbb C^n\setminus S_n(D)$ be such that $\mu_1\in\mathbb C\setminus D$. Then the affine hyperplane 
\begin{multline}
H(w,\mu_1):=\\
\{(\mu_1+ z_1,\mu_1z_1+z_2,\ldots,\mu_1 z_{n-2}+z_{n-1},\mu_1 z_{n-1}):(z_1,\ldots,z_{n-1})\in\mathbb C^{n-1}\}
\end{multline}
passes through $w$ and is disjoint from $S_n(D)$. In particular, $S_n(D)$ is linearly convex.
\end{proposition}
 \begin{proof} Note that
\begin{multline}
H(w,\mu_1)=\\
\{(\mu_1+\pi_{n-1,1}(\lambda^{\prime}),\mu_1\pi_{n-1,1}(\lambda^{\prime})+\pi_{n-1,2}(\lambda^{\prime}),\ldots,\mu_1\pi_{n-1,n-1}(\lambda^{\prime})):\lambda^{\prime}\in\mathbb C^{n-1}\},
\end{multline}
 which follows from the surjectivity of the mapping $\pi_{n-1}:\mathbb C^{n-1}\to\mathbb C^{n-1}$. It follows that $H(w,\mu_1)$ is disjoint from $S_n(D)$. Substituting $\lambda^{\prime}=(\mu_2,\ldots,\mu_n)$ we see that $w=\pi_n(\mu_1,\ldots,\mu_n)\in H(w,\mu_1)$, which finishes the proof.
\end{proof}

\begin{remark}\label{remark-intersection} Let us draw our attention to the following property. If $\mu_1,\ldots,\mu_k\in\mathbb C$ ($1\leq k\leq n$) then the set
\begin{multline}
H(\mu_1,\ldots,\mu_k):=\\
\left\{\pi_n(\mu_1,\ldots,\mu_k,\lambda_1,\ldots,\lambda_{n-k}):\lambda_j\in\mathbb C,\;j=1,\ldots,n-k\right\}
\end{multline}
is an $(n-k)$-dimensional affine space (in the proof of the previous proposition we considered the case $k=1$). Actually, first note that the mapping $\psi:=\pi_n(\mu_1,\ldots,\mu_k,\cdot):\mathbb C^{n-k}\to\mathbb C^n$ is proper. Additionally, the form of $\pi_n$ easily implies that $\psi$ is an affine mapping of variables $\pi_{n-k}(\lambda_1,\ldots,\lambda_{n-k})$ and $\pi_{n-k}:\mathbb C^{n-k}\to\mathbb C^{n-k}$ is also onto. All these facts give the desired property of $H(\mu_1,\ldots,\mu_{k})$. 
\end{remark}

We show how some properties of $D$ induce the same ones of $S_n(D)$ (compare Theorem~\ref{theorem:c-finitely-compact}). The first notion that we discuss is the hyperconvexity.

\begin{proposition} Let $D$ be a domain in $\mathbb C$, $n\geq 2$. Then $S_n(D)$ is hyperconvex iff $D$ is hyperconvex.
\end{proposition}
\begin{proof}  Let $D$ be hyperconvex and let $u:D\to(-\infty,0)$ be a negative subharmonic exhaustion function. Define
\begin{equation}
v(z):=\max\{u(w_j):\pi_n(w_1,\ldots,w_n)=z,\; j=1,\ldots,n\},\; z\in S_n(D).
\end{equation}
The properness of $(\pi_n)_{D^n}$ onto the image and the geometry of $S_n(D)$ imply that $v$ is a negative plurisubharmonic exhaustion function of $S_n(D)$.

To prove the opposite implication fix some $\lambda_1,\ldots,\lambda_{n-1}\in D$ and let $v$ be the negative plurisubharmonic exhaustion function on $S_n(D)$. Let $u(\cdot):=v(\pi_n(\lambda_1,\ldots,\lambda_{n-1},\cdot))$ be defined on $D$. Then $u$ is a negative exhaustion subharmonic function on $D$.
\end{proof}

\subsection{Riemann-type mapping theorem}
For a domain $\Omega\subset\mathbb C^n$ we define {\it the Lempert function} as follows
\begin{equation}
l_{\Omega}(w,z):=\inf\{p(0,\sigma):\exists f\in\mathcal O(\mathbb D,\Omega) \text{ such that } f(0)=w, f(\sigma)=z\}.
\end{equation}
Recall that the Lempert Theorem (see e. g. \cite{Lem 1981}, \cite{Jar-Pfl 2013}) states that if $\Omega$ is convex then $l_{\Omega}\equiv c_{\Omega}$.

In the next result we show a Riemann-type mapping theorem for symmetric powers of planar domains.

\begin{theorem}\label{theorem:riemann-like}
Let $D$ be a bounded, hyperconvex domain in $\mathbb C$, $n\geq 2$. Assume that $c_{S_n(D)}\equiv l_{S_n(D)}$. 
Then $D$ is biholomorphic to $\mathbb D$ and $n=2$.
\end{theorem}
\begin{proof}
Choose pairwise distinct points $\lambda_1^0,\ldots,\lambda_{n-1}^0\in\partial (\operatorname{int}\overline{D})$ (the fact that $D$ is bounded allows us to make such a choice)  and two distinct points $\lambda_n^0,\mu_n^0\in D$. Let us also choose sequences $D\owns \lambda_j^k\to_{k\to\infty} \lambda_j^0$, $j=1,\ldots,n-1$. Now the equality between the Lempert function and the Carath\'eodory distance and the tautness of $S_n(D)$ imply that there exist holomorphic mappings 
\begin{equation}
f_k:\mathbb D\to S_n(D),\; F_k:S_n(D)\to\mathbb D
\end{equation}
such that $F_k\circ f_k=\operatorname{id}_{\mathbb D}$ and $f_k(0)=\pi_n(\lambda_1^k,\ldots,\lambda_{n-1}^k,\lambda_n^0)$, $f_k(\sigma_k)=\pi_n(\lambda_1^k,\ldots,\lambda_{n-1}^k,\mu_n^0)$, where $\sigma_k\in(0,1)$. 

Define 
\begin{equation}
G_k(\lambda):=F_k(\pi_n(\lambda_1^k,\ldots,\lambda_{n-1}^k,\lambda)),\; \lambda\in D.
\end{equation}
Without loss of generality (taking if necessary a subsequence) we have the following convergences (we use here the boundedness of $D$): 
\begin{equation}
\sigma_k\to\sigma_0\in(0,1),\; f_k\to f \text{ and } G_k\to G \text{ locally uniformly},
\end{equation}
where $f:\mathbb D\to\overline{S_n(D)}$,
$G:D\to\mathbb D$, $f(0)=\pi_n(\lambda_1^0,\ldots,\lambda_n^0)$, $f(\sigma)=\pi_n(\lambda_1^0,\ldots,\lambda_{n-1}^0,\mu_n^0)$, $G(\lambda_n^0)=0$, $G(\mu_n^0)=\sigma$.

Define $\tilde g$ (respectively, $\tilde g_k$) to be the $n$ components of the multivalued function $\pi_n^{-1}\circ f$ (respectively, $\pi_n^{-1}\circ f_k$). We know that all the components of $\tilde g$ (respectively, $\tilde g_k$) have values in $\overline{D}$ (respectively, $D$). Additionally at the points $0$ and $\sigma$ all but one components of $\tilde g$ are from $\partial D$. Note that the values of $f$ at these two points are regular values for the proper holomorphic mapping $\pi_n$ and thus the functions $\pi_n^{-1}\circ f$ near these two points ($0$ and $\sigma$) may be chosen to be a holomorphic mapping.

The openness of holomorphic functions and the description of the closure of $\pi_n(D)$ together with the fact that $\lambda_j^0\in\partial (\operatorname{int}\overline{D})$, $j=1,\ldots,n-1$, imply that near these two points all but one components of $\tilde g$ are constant (and equal to $\lambda_1^0,\ldots,\lambda_{n-1}^0$) whereas the last one is from $D$. The fact that some nonempty open part of $f(\mathbb D)$ is lying in the complex line $L:=\{\pi_n(\lambda_1^0,\ldots,\lambda_{n-1}^0,\lambda):\lambda\in\mathbb C\}$ implies easily that $f(\mathbb D)\subset L$ and consequently all but one components of $\tilde g$ are constant (equal to $\lambda_j^0$, $j=1,\ldots,n-1$) and the last component  is a nonconstant holomorpic function $g:\mathbb D\to\operatorname{int}(\overline{D})$ with $g(0)=\lambda_n^0$ and $g(\sigma)=\mu_n^0$. We show below that we have even more; namely, $g(\mathbb D)\subset D$.  In fact, following the same line of argument we get more; the multivalued functions $\tilde g_k$ have the following property: there exists a sequence $r_k\to 1$, $0<r_k<1$, such that the multivalued functions $\tilde g_k$ are actually holomorphic functions when restricted to $r_k\mathbb D$ with values in $D$. Moreover, taking the last component of the multifunction $\tilde g_k$ (which we denote by $g_k$) we get that $g_k\to g$ locally uniformly on $\mathbb D$. Recall that $g_k:r_k\mathbb D\to D$ so the Hurwitz theorem implies that $g(\mathbb D)\subset D$.

Since $G\circ g(0)=0$, $G\circ g(\sigma)=\sigma$, the Schwarz Lemma implies that $G\circ g$ is the identity. Consequently, $D$ is biholomorphic to $\mathbb D$ (with biholomorphisms given by $g$ or $G$). The results on the symmetrized polydisc (see \cite{Cos 2004}, \cite{Agl-You 2004}, \cite{Nik-Pfl-Zwo 2007}) imply that $n=2$.     
\end{proof}

\begin{remark}
It would be interesting to see whether some analogue of the Lempert theorem or the rigidity of the group of automorphims holds for $(\mathbb B_m)_{sym}^n$, 
$m,n\geq 2$ (compare \cite{Agl-You 2004}, \cite{Cos 2004}, \cite{Edi-Zwo 2005}, \cite{Nik-Pfl-Zwo 2007}, \cite{Cha-Grow 2018}). It is also interesting to which extent we could relax assumptions in Theorem~\ref{theorem:riemann-like}. Recall that without some extra assumptions we cannot hope for the implication: 
\begin{equation}
l_{S_n(D)}\equiv c_{S_n(D)}\implies l_D\equiv c_D.
\end{equation}
Namely, in the example $D:=\mathbb C\setminus\{0,1\}$ we have the identities $l_{S_2(D)}\equiv c_{S_2(D)}\equiv 0$ whereas $c_D\equiv 0$ and $l_D(w,z)>0$, $w\neq z$.
\end{remark}

\subsection{Kobayashi hyperbolicity and completeness of symmetric products of planar domains}

Recall that {\it the Kobayashi (pseudo)distance $k_{\Omega}$} of a domain $\Omega\subset\mathbb C^n$ may be defined as the largest pseudodistance smaller than or equal to $l_{\Omega}$. The domain $\Omega$ is called {\it Kobayashi hyperbolic}  if $k_{\Omega}$ is a distance. If additionally, $(\Omega,k_{\Omega})$ is a complete metric space then $\Omega$ is called {\it Kobayashi complete}. Recall that the Kobayashi completeness of a Kobayashi hyperbolic domain is equivalent to {\it the $k$-finite compactness}, i. e. the fact that $k_{\Omega}(z,z^k)\to\infty$ for some (any) $z\in\Omega$ and any sequence $(z^k)_k\subset\Omega$  having no accummulation point (see e. g. \cite{Kob 1998}, \cite{Jar-Pfl 2013}).

We already know that representations of symmetric products of planar domains are linearly convex. It is worth mentioning that a bounded linearly convex domain $\Omega\subset\mathbb C^n$ is automatically Kobayashi complete. We present the proof below.

\begin{proposition}  Let $\Omega\subset\mathbb C^n$ be a bounded linearly convex domain. Then $\Omega$ is Kobayashi complete. 
\end{proposition}
\begin{proof}
Certainly $\Omega$ is Kobayashi hyperbolic.

Fix a boundary point $w\in\partial\Omega$. Let $H$ denote an affine hypersurface passing through $w$ and disjoint from $\Omega$ and let $l$ denote a complex line passing through $w$ orthogonal to $H$. The projection $p$ along $H$ onto $l$ maps $\Omega$ onto the bounded image in $l$ with the point $w$ lying in the boundary of $p(D)$. Let $(z^k)_k\subset D$ be such that $z^k\to w$. Then the contractivity of the Kobayashi pseudodistance gives
\begin{equation}
k_{\Omega}(z^1,z^k)\geq k_{p(\Omega)}(p(z^1),p(z^k)).
\end{equation}  
Since $p(\Omega)$ is a bounded planar domain it is Kobayashi complete. Therefore, the last expression tends to infinity which easily finishes the proof. 
\end{proof}

The above proposition allows us to conclude that a domain $S_n(D)$ is Kobayashi complete if $D\subset\mathbb C$ is bounded. In the unbounded case we should be more careful. Below we present a complete description of Kobayashi hyperbolicity and completeness of symmetric products of planar domains.
We start with the special case.

\begin{proposition}\label{proposition:complement-hyperplanes} Fix $n,N\geq 2$. Let $\mu_1,\ldots,\mu_N\in\mathbb C$ be pairwise different.

If $N\geq 2n$ then the domain $S_n(\mathbb C\setminus\{\mu_1,\ldots,\mu_N\})$ is Kobayashi complete.

If $N<2n$ then the domain $S_n(\mathbb C\setminus\{\mu_1,\ldots,\mu_N\}$ contains a non-constant holomorphic image of $\mathbb C$ and thus it is not Kobayashi hyperbolic.
\end{proposition}
\begin{proof} Simple calculations give the following equality 
\begin{equation}
S_n(\mathbb C\setminus\{\mu_1,\ldots,\mu_N\})=\mathbb C^n\setminus\bigcup_{j=1}^N H_j,
\end{equation}
where (compare Proposition~\ref{proposition:linearly-convex})
\begin{multline}
H_j:=\\
\{(\mu_j+\pi_{n-1,1}(\lambda),\mu_j\pi_{n-1,1}(\lambda)+\pi_{n-1,2}(\lambda),\ldots,\mu_j\pi_{n-1,n-1}(\lambda)):\lambda\in\mathbb C^{n-1}\}=\\
\left\{(\mu_j+z_1,\mu_j z_1+z_2,\ldots,\mu_jz_{n-1}):(z_1,\ldots,z_{n-1})\in\mathbb C^{n-1}\right\},
\end{multline}
$j=1,\ldots,n-1$. Note that the hyperplanes $H_j$ are in general position. In fact, for any $1\leq j_1<\ldots<j_k\leq n$ with $1\leq k\leq N$ we get that
\begin{equation}
H_{j_1}\cap\ldots\cap H_{j_k}=\{\pi_n(\mu_{j_1},\ldots,\mu_{j_k},\lambda_1,\ldots,\lambda_{n-k}):\lambda_l\in\mathbb C\}
\end{equation}
is an $(n-k)$-dimensional affine space (see Remark~\ref{remark-intersection}).


Then the theorem on Kobayashi completeness of the complement of the unions of $(2n+1$) hyperplanes in general position in the projective space (see \cite{Gre 1977}, \cite{Kob 1998}) and results on non-hyperbolicity of complements of $2n$ hyperplanes (see \cite{Kie 1969} and \cite{Snu 1986} or \cite{Kob 1998}) finish the proof.
\end{proof}

\begin{theorem}\label{theorem:kobayashi-complete} Let $D\subset\mathbb C$ be a domain and let $n\geq 2$ be fixed. If $\#(\mathbb C\setminus D)\geq 2n$ then $S_n(D)$ is Kobayashi complete. If $\#(\mathbb C\setminus D)<2n$ then $S_n(D)$ contains a non-constant holomorphic image of $\mathbb C$ and thus $S_n(D)$ is not Kobayashi hyperbolic. 
\end{theorem}
\begin{proof} In view of the previous result it is sufficient to show the first part of the theorem. Let $T\subset\mathbb C\setminus D$ be any set with $2n$-elements. Then $S_n(D)\subset S_n(\mathbb C\setminus T)$ so the contractivity of the Kobayashi pseudodistance implies that $k_{S_n(D)}\geq k_{S_n(\mathbb C\setminus T)}$, which together with the previous result implies the Kobayashi hyperbolicity of $S_n(D)$. To prove the Kobayashi completeness it is sufficient to show that $k_{S_n(D)}(z^1,z^k)\to\infty$ for any sequence $(z^k)_k\subset D$ such that $||z^k||\to\infty$ or $z^k\to z^0\in\partial S_n(D)$. In the first case the result follows from Proposition~\ref{proposition:complement-hyperplanes} (as $k_{S_n(\mathbb C\setminus T)}(z^1,z^k)\to\infty$). In the case 
$z^k\to z^0=\pi_n(\mu_1,\ldots\mu_n)$, where $\mu_1\in\partial D$, $\mu_j\in\overline{D}$, $j=2,\ldots,n$ we choose a set $T\subset\overline{D}$ having $2n$ elements such that $\mu_j\in T$, $j=1,\ldots,n$ which is possible due to the assumptions. Then $D\subset \mathbb C\setminus T$ and $z^0\in\partial S_n(\mathbb C\setminus T)$ so
\begin{equation}
k_{S_n(D)}(z^1,z^k)\geq k_{S_n(\mathbb C\setminus T)}(z^1,z^k).
\end{equation}
And the last expression tends to infinity by Proposition~\ref{proposition:complement-hyperplanes}, which finishes the proof.
\end{proof}

\begin{remark} As we saw in the proof of Theorem~\ref{theorem:kobayashi-complete} the description of Kobayashi complete symmetric products of planar domains relied not only on the linear convexity of $S_n(D)$ but also on the special geometry of $S_n(D)$. It could be interesting to see whether the following could be true: a linearly convex domain, which admits a certain number (at least $2n$) of hyperplanes in a general position disjoint from the domain, is Kobayashi complete.

\end{remark}

\subsection{Carath\'eodory hyperbolicity}
It turns out that in the class of symmetric products of planar domains the Carath\'eodory hyperbolicity is preserved under taking symmetric powers.

\begin{proposition}\label{proposition:caratheodory-hyperbolic} Let $D$ be a domain in $\mathbb C$. Then $D$ is Carath\'eodory hyperbolic if and only if $S_n(D)$ is Carath\'eodory hyperbolic. Consequently, $D$ is $c$-finitely compact if and only if $S_n(D)$ is $c$-finitely compact.
\end{proposition}
\begin{proof} It is sufficient to show that $c$-hyperbolicity of $D$ implies that of $S_n(D)$. Assume that $D$ is Carath\'eodory hyperbolic. Let $\pi_n(\lambda_1,\ldots,\lambda_m)\neq \pi_n(\mu_1,\ldots,\mu_n)$ with $\lambda_j,\mu_j\in D$, $j=1,\ldots,n$. Then without loss of generality we may assume that 
$\lambda_1\not\in\{\mu_1,\ldots,\mu_n\}=\{x_1,\ldots,x_k\}$. Then we claim that there is an $f\in\mathcal O(D,\mathbb D)$ with $f(\lambda_1)=0$, $f(\mu_j)\neq 0$. In fact, the Carath\'eodory hyperbolicity implies that there is a non-constant bounded $g\in\mathcal O(D)$ such that $g(\lambda_1)=0$. Then the function $h(\cdot):=\frac{g(\cdot)}{\prod_{j=1}^k(\cdot-x_j)^{r_j}}$, where $r_j$ is the multiplicity of $g$ at $x_j$, is a bounded holomorphic function with $h(\lambda_1)=0$ and $h(\mu_j)\neq 0$ which gives the claim. Take the function $f$ from the claim. Then
$z:=\pi_n(f(\lambda_1),\ldots,f(\lambda_n))\neq \pi_n(f(\mu_1),\ldots,f(\mu_n))=:w$, so
\begin{equation}
c_{S_n(D)}(\pi_n(\lambda_1,\ldots,\lambda_n),\pi_n(\mu_1,\ldots,\mu_n))\geq c_{\mathbb G_n}(z,w)>0.
\end{equation}

\end{proof}

\begin{remark}
Recall that in the class of planar domains by a recent result (Theorem 1 in \cite{Edi 2018}) two closely related notions of Carath\'eodory completeness and $c$-finite compactness are equivalent. Moreover, they are both equivalent to the fact that any boundary point $z$ of $D$ is an $A(D,\{z\})$ peak point. Note that although the $c$-finite compactness is equivalent to $c$-finite compactness of $S_n(D)$ (Proposition~\ref{proposition:caratheodory-hyperbolic}) we did not prove the equivalence of $c$-finite compactness of $S_n(D)$ with the fact that any boundary point $z$ of $S_n(D)$ is a weak $A(D,\{z\})$ point.
\end{remark}

\end{document}